\documentclass[12pt,reqno,twoside]{amsart}
\usepackage{amsmath}
\usepackage{amssymb}
\usepackage{color}
\usepackage{epsfig}
\usepackage{color}
\title[Remarks on minimal sets]{Remarks on minimal sets and  conjectures of Cassels,
  Swinnerton-Dyer, and Margulis}
\author{Jinpeng An}
\address{LMAM, School of Mathematical Sciences and BICMR, Peking University, Beijing, 100871, China
{\tt anjinpeng@gmail.com}}
\author{Barak Weiss}
\address{Ben Gurion University, Be'er Sheva, Israel 84105
{\tt barakw@math.bgu.ac.il}}

\newif\ifdraft\drafttrue

\newcommand{\goth}[1]{{\mathfrak{#1}}}
\newcommand{\Q}{{\mathbb {Q}}}

\newcommand{\R}{{\mathbb{R}}}

\newcommand{\OO}{{\mathcal{O}}}

\newcommand{\Z}{{\mathbb{Z}}}

\newcommand{\C}{{\mathbb{C}}}

\newcommand{\N}{{\mathbb{N}}}

\newcommand{\cl}{\overline}

\newcommand{\Ad}{{\operatorname{Ad}}}

\newcommand{\GL}{\operatorname{GL}}

\newcommand{\SL}{\operatorname{SL}}

\newcommand{\PGL}{\operatorname{PGL}}

\newcommand{\Lie}{\operatorname{Lie}}

\newcommand{\diag}{{\rm diag}}

\newcommand{\Gal}{{\rm Gal}}
\newcommand {\ignore}[1]  {}

\newcommand{\df}{{\, \stackrel{\mathrm{def}}{=}\, }}

\newcommand{\til}{\widetilde}

\newcommand{\sm}{\smallsetminus}

\newcommand{\vre}{\varepsilon}

\newcommand{\Stab}{\mathrm{Stab}}

\newtheorem{thm}{Theorem}

\newtheorem{lem}[thm]{Lemma}

\newtheorem{prop}[thm]{Proposition}

\newtheorem{cor}[thm]{Corollary}

\newtheorem{conj}[thm]{Conjecture}

\theoremstyle{definition}

\begin{document}

\maketitle

\begin{abstract}
We prove that a hypothesis of Cassels, Swinnerton-Dyer, recast by
Margulis as statement on the action of the
diagonal group $A$ on the space of unimodular lattices, is equivalent to several
assertions about minimal sets for this action. More generally, for a
maximal $\R$-diagonalizable subgroup $A$ of a
reductive group $G$ and a lattice $\Gamma$ in $G$, we give a sufficient condition for a
compact $A$-minimal subset $Y$ of $G/\Gamma$ to be of a simple form, which is also necessary if $G$ is $\R$-split.
We also show that the stabilizer of $Y$ has no nontrivial connected unipotent subgroups.
\end{abstract}

\section{Introduction}

Given a group $A$ acting on a locally compact space $X$, we say that
the action is {\em minimal} if every $A$-orbit is dense, or
equivalently, if there are no proper $A$-invariant closed subsets. We
say that a subset $Y \subset X$ is a {\em minimal set} if it is
$A$-invariant and closed, and the restriction of the $A$-action to
$Y$ is minimal. If
$X$ is compact then any closed invariant subset contains a minimal
set, and the study of minimal sets is often an important first step in
the study of all $A$-orbits in $X$. The study of minimal sets may
be viewed as the topological counterpart of the study of
invariant ergodic measures. In many dynamical systems,
there are many minimal sets defying a simple classification; however,
as we will see,
in some dynamical systems of algebraic origin, the classification of
minimal sets might be possible and is of great interest.

This paper is devoted to the study of minimal sets for the action of
$A$ on $G/\Gamma$, where $G$ is a connected reductive real
algebraic group, $\Gamma$ is a lattice (i.e. a discrete subgroup such
that the quotient $G/\Gamma$ supports a finite $G$-invariant measure),
and $A$ is a maximal connected $\R$-diagonalizable subgroup of $G$. Recall that
a {\em root} on $A$ is a nontrivial homomorphism $\alpha: A \to \R^*$
such that there is a nonzero  $\mathbf{v}\in \goth{g} = \Lie(G) $ such that for any $a \in
A$, $\Ad(a)\mathbf{v} = \alpha(a)\mathbf{v}$, where $\R^*$ denotes the multiplicative group of positive real
numbers. We will denote the kernel of $\alpha$ by
$A_\alpha$; this is a subgroup of $A$ which is of codimension one. We will pay particular
attention to the action of the groups $A_\alpha$ on a minimal set
for $A$.

We begin with the important special case in which
$G = \SL_n(\R), \, n \geq 3$,  $\Gamma = \SL_n(\Z)$, and $A$ is the subgroup of all
diagonal matrices with positive diagonal entries. In this case the space $G/\Gamma$ is
naturally identified with the space of unimodular lattices in $\R^n$,
where the action is simply the linear action of a matrix as a linear
isomorphism of $\R^n$. The dynamics of the
$A$-action on $G/\Gamma$ has been intensively studied, both for its
intrinsic interest as a prototypical action for which one may hope to
classify all minimal sets and invariant measures,
and for its connections to number-theoretic questions.
See \cite{EL}
for a survey of the history as well as many recent developments.

To illustrate this connection, let $f$ be the product of three real
linear forms in three variables. In \cite[Hypothesis A]{CaSD}, Cassels and
Swinnerton-Dyer asked whether
$
\inf_{\vec{x}\in \Z^3 \sm \{0\}} |f(\vec{x})|>0 $ implies that $f$ is
a multiple of a form with integer coefficients\footnote{Actually the
  statement above is the negation of what Cassels and
  Swinnerton-Dyer called `Hypothesis A'.}.
They showed that an affirmative answer to this question would yield an
affirmative solution to a famous conjecture of Littlewood, and
formulated several additional questions about products of three linear
real forms, which remain open to this day.
Cassels and Swinnerton-Dyer did not make an explicit
conjecture regarding the correct answer to their questions, but did
express the opinion that `we  tend
to believe [that Hypothesis A is true]'. In 2000 Margulis
\cite{Margulis_conjecture} showed the equivalence of Hypothesis A with
a dynamical statement, which he stated explicitly as a conjecture:

\begin{conj}[Margulis] \label{conj: margulis}
Let $n \geq 3, \, G = \SL_n(\R), \, \Gamma = \SL_n(\Z), $ and let $A$
be the group of positive diagonal matrices.
Any $A$-orbit on $G/\Gamma$ which is bounded (i.e. has compact
closure) is closed (and hence compact).
\end{conj}

In \S \ref{section: conjectures} we will show that Conjecture
\ref{conj: margulis} is equivalent to statements regarding minimal
sets for the $A$-action. Namely we will show:

\begin{thm}\label{thm: equivalence of conjectures} Let $G, \Gamma, A$
  be as in Conjecture \ref{conj: margulis}.
The following assertions are equivalent to Conjecture \ref{conj: margulis}:
\begin{itemize}
\item[(a)]
Any compact minimal set for the $A$-action on $G/\Gamma$ is a compact
orbit.
\item[(b)]
For any root $\alpha$ and any compact $A$-minimal set $Y\subset
G/\Gamma$, $Y$ is also minimal for the action of $A_\alpha$.
\item[(c)]
For any root $\alpha$ and any compact minimal set $Y \subset G/\Gamma$ for the action of
$A$, any $A_{\alpha}$-minimal subset $Y_\alpha \subset Y$ satisfies
$AY_{\alpha} =Y$.
\item[(d)]
For any root $\alpha$ and any compact minimal set $Y \subset G/\Gamma$ for the action of
$A$, there exists an $A_{\alpha}$-minimal subset $Y_\alpha \subset Y$
such that $AY_{\alpha} =Y$.
\end{itemize}
\end{thm}

Now we turn to the general case. The assertion of Conjecture
\ref{conj: margulis} is not true for general $G/\Gamma$. In fact, in
unpublished work, Mary Rees showed that it does not hold even for
$G=\SL_3(\R)$ and $\Gamma$ a certain cocompact lattice. An interesting
feature of Rees' examples is that the groups $A_{\alpha}$ do not act
minimally on $A$-minimal sets, see the discussion in \cite[\S5]{LW}. That
is, Rees' example also shows that condition (b) of Theorem \ref{thm: equivalence
  of conjectures} need not hold in general. In this paper we clarify
the implications among the conditions (a--d), in the context of a
general $G/\Gamma$; among other things, it will develop that it is no coincidence, that
for certain $G/\Gamma$, the same constructions show the failure of
Conjecture \ref{conj: margulis} and condition (b). More precisely, we will prove the following result.

\begin{thm}\label{T:main2}
Let $G$ be a connected reductive real algebraic group, $\Gamma$ a lattice in $G$, $A$ a maximal connected $\R$-diagonalizable subgroup of $G$, and $Y \subset G/\Gamma$ a compact $A$-minimal set. Suppose that for any root $\alpha$, there exists an $A_\alpha$-minimal
set $Y_\alpha\subset Y$ such that $AY_\alpha=Y$. Then there exists a compact subtorus $T$ of $G$ centralized by $A$ such that $Y$ is an $AT$-orbit.
\end{thm}

Note that if $G$ is $\R$-split, then $A$ is of finite index in $Z_G(A)$, and there is no nontrivial compact torus centralized by $A$. So we have the following corollary, which implies the equivalence of (a), (c) and (d) in Theorem \ref{thm: equivalence of conjectures}.

\begin{cor}\label{cor}
Under the conditions of Theorem \ref{T:main2}, suppose in addition that $G$ is $\R$-split, then the
following are equivalent:
\begin{itemize}
\item[(a)]
$Y$ is a single (compact) $A$-orbit.
\item[(b)]
For any root $\alpha$, any $A_\alpha$-minimal
set $Y_\alpha\subset Y$ satisfies $AY_\alpha=Y$.
\item[(c)]
For any root $\alpha$, there exists an $A_\alpha$-minimal
set $Y_\alpha\subset Y$ such that $AY_\alpha=Y$.
\end{itemize}
\end{cor}

For a Lie group $H$, we denote its connected component of the identity by $H^0$. For any compact $A$-minimal set $Y \subset G/\Gamma$, denote
$$
\Stab_G(Y) \df \{g \in G: gY=Y\}.
$$
The proof of Theorem \ref{T:main2} is based on the following result, which is of interest in itself.

\begin{thm}\label{T:main1}
Let $G, \Gamma, A$ be as in Theorem \ref{T:main2}, and let $Y\subset G/\Gamma$ be a compact $A$-minimal set.
Then $\Stab_G(Y)$ has no nontrivial connected unipotent subgroups. Equivalently, we have
$\Stab_G(Y)^0\subset Z_G(A)$.
\end{thm}

If $G$ is $\R$-split, Theorem \ref{T:main1} asserts that $\Stab_G(Y)^0=A$, i.e., $A$ is of finite index in $\Stab_G(Y)$. However, in the general case, it may happen that $\Stab_G(Y)$ contains a full Cartan subgroup of $G$ which is strictly larger than $A$. We will illustrate this with an example where $G$ is the underlying real algebraic group of $\SL_2(\C)$ at the end of Section 3.

The proof of Theorem \ref{T:main1} relies on work of Dani and Margulis
\cite{DM} and Prasad and Raghunathan \cite{PR}. It also relies on a
statement (Proposition \ref{prop: minimal normalized} below) about orbit-closures of elements of an
$A$-minimal set, under a unipotent group normalized by $A$. The proof
of Proposition \ref{prop: minimal normalized} employs some ideas of
Mozes \cite{Mozes epimorphic}. Throughout this paper, by a real algebraic group, we mean an open subgroup of the Lie group of real points of an algebraic group defined over $\R$.

\subsection{Acknowledgements}
J.A. acknowledges support of NSFC grant 10901005 and FANEDD grant
200915. B.W. acknowledges support of European Research Council grant
DLGAPS 279893, and  Israel Science Foundation grant 190/08. J.A. is grateful to the hospitality of Ben Gurion University where this
work was initiated, and to Jiu-Kang Yu for helpful discussion on Galois theory involved in the proof of Proposition \ref{example}. 
Some of our work on this paper were conducted in Astrakhan, Russia, during the conference {\em Diophantine analysis} in summer 2012, 
and we would like to thank the organizers of the
conference for a stimulating environment. We would also like to thank
Elon Lindenstrauss and Uri Shapira for useful comments.

\section{Conjectures for the $A$-action}\label{section: conjectures}

We begin with a simple
proof of Theorem \ref{thm: equivalence of conjectures}, based on the
results of \cite{LW}. We will need the following:

\begin{prop}\label{prop: CSD unbounded}
Let $G/\Gamma = \SL_n(\R)/\SL_n(\Z)$, let $A$ be the group of diagonal
positive matrices, and let $U$ be a one-parameter unipotent
subgroup of $G$ normalized by $A$. Then any $AU$-orbit in $G/\Gamma$ is
unbounded.
\end{prop}

\begin{proof}
This was proved in \cite{CaSD} and again in \cite{LW}; we repeat the proof for
completeness.

Recall Mahler's criterion, which asserts that $X \subset G/\Gamma$ is
bounded if and only if there is $\vre>0$ such that any nonzero vector
in any lattice in $X$ has length at least $\vre$. Since $A$ normalizes
$U$, there are distinct $i,j \in \{1, \ldots, n\}$ such that, if we define
$u(s)= \exp(sE_{ij})$, then $U = \{u(s): s \in \R\}$, where $E_{ij}$ is the matrix with 1 in the
$(i,j)$th entry and zero elsewhere. To simplify notation assume $j=1$.
Let $x \in G/\Gamma$, and let $v
= (v_1, \ldots, v_n)$ be a nonzero vector in the lattice corresponding
to $x$. If $v_1=0$ then we can apply the elements  $a(t) \df \diag(e^{(n-1)t},
e^{-t}, \ldots, e^{-t})$, and we will have $a(t)v \to_{t \to \infty} 0$, implying via
Mahler's compactness criterion
that $Ax$ is unbounded. If $v_1 \neq 0$ then we can apply
$u(-v_i/v_1)$ to $x$ to obtain, in the lattice corresponding to
$u(-v_i/v_1)x$, a vector with vanishing $i$-th entry. Now we repeat the
previous step to this vector.
\end{proof}

\begin{proof}[Proof of Theorem \ref{thm: equivalence of conjectures}]
Conjecture \ref{conj: margulis} $\iff$ (a): Suppose
Conjecture \ref{conj: margulis} holds. Any orbit inside any compact minimal
set $Y$ is bounded, hence a compact orbit, which by
minimality coincides with $Y$.
Conversely, assume (a), and suppose $Ax$ is a
bounded orbit. Then $X_0 \df \cl{Ax}$ contains a minimal set and hence
a compact orbit. Now an argument of \cite{CaSD} (see
also \cite{Margulis_oppenheim, LW}) shows that if $X_0$ is not itself
a compact $A$-orbit, then it
contains an $AU$-orbit and hence cannot
be compact, in view of Proposition \ref{prop: CSD unbounded}.

\medskip

(a) $\implies$ (b): This is proved in \cite[Step 6.1]{LW}.

\medskip

(b) $\implies$ (c) $\implies$ (d): Immediate.

\medskip

(d) $\implies$ (c): Suppose $Y_{\alpha}$ is minimal with $AY_{\alpha}
= Y$. Then the sets $\{aY_{\alpha}:a\in A\}$ are all $A_\alpha$-minimal, and they cover
$Y$. Thus if $Y' \subset Y$ is another $A_{\alpha}$-minimal set, then
there is $a \in A$ such that $Y'=aY_\alpha$, and hence $AY' = Y$.

\medskip

(c) $\implies$ (a): If (a) is false then there is an $A$-minimal set
$Y$ which does not contain a compact $A$-orbit. Let $a_0 \in A$ be a
regular element, i.e. $a_0 \in A \sm \bigcup_\alpha A_\alpha$. For
any root $\alpha$, let $\mathbf{u}_\alpha$ be a generator of the root space
$$
\goth{u}_\alpha \df \{\mathbf{v} \in \goth{g}: \forall a \in A, \, \Ad(a)\mathbf{v} = \alpha(a)\mathbf{v}\},
$$
and let $u_\alpha(s) = \exp(s\mathbf{u}_\alpha), \, U_{\alpha} = \{u_\alpha(s): s \in \R\}.$ Let $U^+$ be the
group generated by $\{U_{\alpha}: \alpha(a_0)>0\}$. Arguing as in
\cite[Steps 4.3, 4.4]{LW} (up to replacing $a_0$ by $a_0^{-1}$), we
see that there are distinct $x_1, x_2 \in Y$ and $u^+ \in U^+\sm \{e\}$ such
that $x_2 = u^+x_1$. Arguing as in \cite[Step 4.5]{LW}, there is a root
$\alpha$, distinct $y_1, y_2 \in Y$ and $u=u_\alpha(s_0) \in U_{\alpha},
\, s_0\neq 0$ such that
$y_2 = uy_1$.  In light of Proposition \ref{prop: CSD
  unbounded}, it suffices to show that $Y$ contains a
$U_{\alpha}$-orbit.

Since $AY_{\alpha} = Y$, and each $aY_{\alpha}$ is
$A_{\alpha}$-minimal, we can assume that $y_1 \in Y_\alpha$,
and we claim that $y_2 \in Y_{\alpha}$.
If not then $y_2
\in Y' \df  a'Y_{\alpha}$ for some $a' \in A \sm A_\alpha$, where $Y'$
is also $A_\alpha$-minimal, and there
is a positive distance between $Y'$ and $Y_\alpha$.
According to Lemma \ref{L:cocompact}, which we prove below in a more
general setting, there is $a \in A \sm A_{\alpha}$ such that $a
Y_\alpha = Y_\alpha$.
Replacing $a$ with $a^{-1}$ if necessary, we may assume that $\alpha(a) <1.$
Then for each $k \in \N$,
\[
Y' \ni a^ky_2 = a^kuy_1=\exp(\alpha(a)^ks_0\mathbf{u}_\alpha)a^ky_1 \in\exp(\alpha(a^k)s_0\mathbf{u}_\alpha)Y_{\alpha}.
\]
Since $\alpha(a)^k \to_{k \to \infty} 0$, the distance from $Y'$ to $Y_{\alpha}$ is bounded
above by a quantity tending to zero as $k \to \infty$, a contradiction proving the claim.

Since $y_2 = uy_1 \in Y_{\alpha}$, and the $A_{\alpha}$-action
commutes with that of $u$ we find that $H \df \Stab_G(Y_{\alpha})$
contains $u$. Since $H$ is a closed group, it contains all
conjugates of $u$ by $a$ and hence contains
$U_{\alpha}$. This completes the proof.
\end{proof}

\section{Proof of Theorem \ref{T:main1}}

We will need the following statement, which is proved in \cite{SW} in
a more restricted setting:

\begin{prop}\label{prop: reductive hom spaces}
Let $G$ be a real algebraic group, $\Gamma$ a lattice in $G$, $L$ a connected
reductive algebraic subgroup of $G$, and $y\in G/\Gamma$. Suppose that the orbit
$Ly$ is closed. Then $L$ can be decomposed as a direct product $L=Z\times L'$ of its connected closed normal subgroups such that
\begin{itemize}
\item[(i)] $L'$ contains $L_y\df\{\ell\in L:\ell y=y\}$ as a lattice, and contains the maximal connected normal subgroup of $G$ with compact center.
\item[(ii)] $Z$ is an $\R$-diagonalizable algebraic subgroup of the center of $L$.
\item[(iii)] The natural map $$Z\times(L'/L_y)\to Ly, \qquad (z,\ell L_y)\mapsto z\ell y$$ is a homeomorphism.
\end{itemize}
\end{prop}

\begin{proof}
Without loss of generality, we may assume that $y$ is the projection of $e$ in $G/\Gamma$, and hence $L_y=L\cap\Gamma$.
Let $S$ be the maximal connected normal semisimple subgroup of $L$ without compact factors. By Ratner's theorem, there is a connected closed subgroup $S'$ of $G$ containing $S$ such that $\cl{Sy}=S'y$ and $S'\cap\Gamma$ is a lattice in $S'$. Since $Ly$ is closed, we have $S'\subset L$. Let $K$ be the maximal connected normal compact subgroup of $L$, and let $N=KS'$. Then $N$ is a connected closed normal subgroup of $L$ and $L/N$ is a simply connected abelian Lie group. Since $S'y$ is closed in $Ly$, $S'L_y$, and hence $NL_y$, is closed in $L$. Thus $q(L_y)$ is a discrete subgroup of $L/N$, where $q:L\to L/N$ is the quotient homomorphism. Let $C$ be the unique connected subgroup of $L/N$ which contains $q(L_y)$ as a lattice, and let $L'=q^{-1}(C)$. We prove that $L'$ satisfies (i). It suffices to prove that $L_y$ is a lattice in $L'$. Firstly, since $q(L_y)$ is a lattice in $C$, $L'/NL_y$ carries an $L'$-invariant finite measure. On the other hand, since $S'\cap\Gamma$ is a lattice in $S'$, it follows that $N\cap\Gamma$ is a lattice in $N$, and hence $L_y$ is a lattice in $NL_y$ (see \cite[Lem. 1.7]{Rag}). Thus $L_y$ is a lattice in $L'$ (see \cite[Lem. 1.6]{Rag}).

Since $L'$ satisfies (i), we can choose $Z$ satisfying (ii) such that $L=Z\times L'$. Moreover, the natural map in (iii) is the composition of the natural homeomorphisms $Z\times(L'/L_y)\cong L/L_y$ and $L/L_y\cong Ly$, hence is a homeomorphism. This completes the proof.
\end{proof}

In the rest of this paper, we assume that $G, \Gamma, A$ are as in Theorem \ref{T:main2}, i.e., $G$ is a connected reductive real algebraic group, $\Gamma\subset G$ is a lattice, and $A\subset G$ is a maximal connected $\R$-diagonalizable subgroup. Let $\Phi=\Phi(A,G)$ be the root system. For $\alpha\in\Phi$, let $\goth{u}_\alpha\subset\goth{g}$ be the root space of $\alpha$, and let $U_\alpha=\exp(\goth{u}_\alpha)$ be the root group.
The following statement will be important for the sequel, and is of
independent interest. Its proof is inspired by ideas of Shahar
Mozes \cite{Mozes epimorphic}.

\begin{prop}\label{prop: minimal normalized}
Let $Y \subset G/\Gamma$ be a compact $A$-minimal set, $U$ be a connected unipotent subgroup of $G$ normalized
by $A$. Then there exist an $A$-invariant residual subset $Y_0 \subset Y$ and a connected reductive closed subgroup $\til{H}$ of $G$ which contains $U$ and is normalized by $A$ such that
\begin{itemize}
\item[(i)] For every $y\in Y_0$, we have $\cl{Uy}=\til{H}y.$
\item[(ii)] The algebraic subgroup $L\df N_G(\til{H})^0$ of $G$ is reductive, and $Y$ is contained in a closed $L$-orbit of finite volume.
\end{itemize}
\end{prop}

\begin{proof}
We divide the proof into steps.

\textbf{Step 1.} Let $\pi:G \to G/\Gamma$ be the natural projection, and let
$\mathcal{H}$ be the collection of connected closed subgroups $H$ of $G$ for
which $H\pi(e) \cong H/\Gamma \cap H$ is closed and of finite volume, and such that the
subgroup of $H$ generated by its
unipotent elements acts ergodically on $H\pi(e)$. It was shown by
Ratner (see \cite{Ratner ICM}) that $\mathcal{H}$ is
countable, and for any $g\in G$, there is $H \in \mathcal{H}$
such that $\cl{U\pi(g)} = (gHg^{-1})\pi(g)$. For any $H \in \mathcal{H}$, let
$$X(H,U) \df \{g \in G:  Ug \subset gH\}.$$
Each $X(H,U)$ is an algebraic subvariety of $G$,
and for any $g \in X(H,U)$ we have $$\cl{U\pi(g)} \subset gH\pi(e) = (gHg^{-1})\pi(g).$$
Note that $$N_G(U)X(H,U)N_G(H)=X(H,U).$$ In particular, $X(H,U)$ is left $A$-invariant. Note also that $G \in \mathcal{H}$, and thus $G/\Gamma = \bigcup_{H \in \mathcal{H}}\pi(X(H,U)).$

Let $H_0\in\mathcal{H}$ be a group of minimal dimension
  for which $\pi(X(H_0,U))$ contains an open subset of $Y$. Since $Y$
  is $A$-minimal and $\pi(X(H_0,U))$ is $A$-invariant, we have
$$Y \subset \pi(X(H_0,U)).$$
By the choice of $H_0$, if $H\in\mathcal{H}$ is a proper subgroup of $H_0$, then  $Y \sm \pi(X(H,U))$ is a
residual subset of $Y$, and hence so is
$$
Y_0 \df Y \sm \bigcup_{H\in\mathcal{H},H\subsetneqq H_0} \pi(X(H,U)).
$$
Note that $Y_0$ is $A$-invariant. It is easy to see that for $y\in Y$ and $g\in X(H_0,U)$ with $y=\pi(g)$, $y\in Y_0$ if and only if $\cl{Uy} = (gH_0g^{-1})y$.
In the sequel we prove that $Y_0$ has the desired properties.

\textbf{Step 2.} We prove that for any $y\in Y_0$ and $g\in X(H_0,U)$ with $y=\pi(g)$, there is a neighborhood $\Omega$ of $g$ in $G$ such that $$\pi^{-1}(Y)\cap\Omega\subset X(H_0,U).$$ If not, then there is a sequence $\{g_n\}\subset \pi^{-1}(Y)\sm X(H_0,U)$ with $g_n\to g$. By the Baire category theorem, there is a compact set $K\subset X(H_0,U)$ such that $\pi(K)$ contains a nonempty open subset of $Y$. In view of the minimality of the $A$-action on $Y$ and the left $A$-invariance of $X(H_0,U)$, we may assume that $\pi(K)$ contains a neighborhood of $y$ in $Y$. Since $Y\ni\pi(g_n)\to y$, we may assume that $\pi(g_n)\in\pi(K)$. So there are sequences $\{k_n\}\subset K$ and $\{\gamma_n\}\subset\Gamma$ such that $g_n=k_n\gamma_n$. Since $K$ is compact, we may also assume that $k_n\to k$ for some $k\in K$. Hence $\gamma_n\to\gamma$, where $\gamma=k^{-1}g\in\Gamma$. It follows that $\gamma_n=\gamma$ for all large $n$. Since $\pi(g)=\pi(k)=y\in Y_0$, we have
$$(gH_0g^{-1})y=\cl{Uy}=(kH_0k^{-1})y.$$
This implies that $\gamma\in N_G(H_0)$. Thus for all large $n$ we have $$g_n=k_n\gamma\in X(H_0,U)N_G(H_0)=X(H_0,U),$$ a contradiction.

\textbf{Step 3.} Let $V=\bigoplus_{k=1}^{\dim G}\bigwedge^k\goth{g}$, and $\bar{V}$ be the projective space of $V$. Consider the representation $\rho=\bigoplus_{k=1}^{\dim G}\bigwedge^k\Ad:G\to\GL(V)$ and the corresponding projective representation  $\bar{\rho}: G \to \PGL(\bar{V})$. If $S$ is a Lie subgroup of $G$, we denote $\bar{\mathbf{p}}_S=\bigwedge^{\dim S}\Lie(S)$, which is a line in $\bigwedge^{\dim S}\goth{g}$, hence an element in $\bar{V}$. By Ratner's theorem, for every $x\in G/\Gamma$, there is a connected closed subgroup $S$ of $G$ containing $U$ such that $\cl{Ux}=Sx$. This defines a map $$\varphi:G/\Gamma\to\bar{V}, \qquad \varphi(x)=\bar{\mathbf{p}}_S.$$ Note that if $\cl{Ux}=Sx$, then $\cl{U(ax)}=a\cl{Ux}=(aSa^{-1})(ax)$. This means that $\varphi$ is $A$-equivariant, i.e.,
$$\varphi(ax)=\bar{\rho}(a)\varphi(x), \qquad \forall a\in A, x\in G/\Gamma.$$
We prove below that the map $\varphi$ is continuous on $Y_0$.

Let $y_n,y\in Y_0$ with $y_n\to y$. We need to show that $\varphi(y_n)\to\varphi(y)$. Choose $g\in X(H_0,U)$ with $y=\pi(g)$, and choose $g_n\in G$ with $y_n=\pi(g_n)$ and $g_n\to g$. By Step 2, we may assume that $g_n\in X(H_0,U)$ for every $n$. Since $y_n,y\in Y_0$, we have
$$\cl{Uy_n}=(g_nH_0g_n^{-1})y_n \qquad \text{and} \qquad \cl{Uy}=(gH_0g^{-1})y.$$
This means that $$\varphi(y_n)=\bar{\rho}(g_n)\bar{\mathbf{p}}_{H_0} \qquad \text{and} \qquad \varphi(y)=\bar{\rho}(g)\bar{\mathbf{p}}_{H_0}.$$
Hence $\varphi(y_n)\to\varphi(y)$.

\textbf{Step 4.} We prove that there is a connected closed subgroup $\til{H}$ of $G$ which contains $U$ and is normalized by $A$ such that for every $y\in Y_0$, we have $\cl{Uy}=\til{H}y.$ Note that a connected subgroup $S$ of $G$ is normalized by $A$ if and only if $\bar{\mathbf{p}}_S$ is fixed by $\bar{\rho}(A)$. Thus it suffices to prove that the map $\varphi$ is constant on $Y_0$ and $\varphi(Y_0)\subset\bar{V}^A$, where $\bar{V}^A$ is the set of $\bar{\rho}(A)$-fixed points in $\bar{V}$.

Let $y\in Y_0$. If $\varphi(y)\notin \bar{V}^A$, then there is a neighborhood $N$ of $\varphi(y)$ in $\bar{V}$ such that $\cl{N}\cap\bar{V}^A=\emptyset$. Let
$$A_0 \df \{a\in A:ay\in \varphi^{-1}(N)\}.$$
Since $\varphi$ is $A$-equivariant, we have $\bar{\rho}(A_0)\varphi(y)\subset N$. Hence $$\cl{\bar{\rho}(A_0)\varphi(y)}\cap\bar{V}^A=\emptyset.$$
On the other hand, since the $\rho(A)$-action on $V$ is $\R$-diagonalizable, there is a finite set $\Psi$ of homomorphisms $A\to\R^*$ and a direct sum decomposition
$$V=\bigoplus_{\chi\in\Psi}V_\chi, \quad \text{where} \quad V_\chi=\{v\in V:\rho(a)v=\chi(a)v,\forall a\in A\}.$$
Note that $\bar{V}^A$ is equal to the union of the projective spaces of $V_\chi$.
Choose a nonzero vector $v\in\varphi(y)$, and write $v=\sum_{\chi\in\Psi'}v_\chi$, where $\emptyset\ne\Psi'\subset\Psi$ and $0\ne v_\chi\in V_\chi$ for every $\chi\in\Psi'$. Since $\varphi$ is continuous on $Y_0$, $\varphi^{-1}(N)$ contains a neighborhood of $y$ in $Y_0$.
By the minimality of the $A$-action on $Y$, $A_0$ is syndetic in $A$, i.e., there is a compact subset $C\subset A$ with $A_0C=A$. From this it is easy to see that there exist $\chi_0\in\Psi'$ and a sequence $\{a_n\}\subset A_0$ such that $(\chi_0^{-1}\chi)(a_n)\to 0$ for any $\chi\in\Psi'\sm\{\chi_0\}$. Thus
$$\chi_0(a_n)^{-1}\rho(a_n)v=v_{\chi_0}+\sum_{\chi\in\Psi'\sm\{\chi_0\}}(\chi_0^{-1}\chi)(a_n)v_\chi\to v_{\chi_0}.$$
This means that $\bar{\rho}(a_n)\varphi(y)$ converges to a line in $V_{\chi_0}$, which is a point in $\bar{V}^A$. Thus $\cl{\bar{\rho}(A_0)\varphi(y)}\cap\bar{V}^A\ne\emptyset$. This is a contradiction, hence proves that $\varphi(Y_0)\subset\bar{V}^A$.

Since $\varphi$ is continuous on $Y_0$ and is $A$-equivariant, and every $A$-orbit in $Y_0$ is dense in $Y_0$, it follows that every $\bar{\rho}(A)$-orbit in $\varphi(Y_0)$, which is a single point, is dense in $\varphi(Y_0)$. Thus $\varphi(Y_0)$ consists of only one point. Hence $\varphi$ is constant on $Y_0$.

\textbf{Step 5.} We prove that $\til{H}$ is reductive. For a Lie subgroup $S$ of $G$, denote $$N_G^1(S)\df\{g\in N_G(S):\det\Ad(g)|_{\Lie(S)}=1\}.$$
We first prove that $A\subset N_G^1(\til{H})$. Consider the character $\chi:A\to\R^*$ defined by  $\chi(a)=\det\Ad(a)|_{\Lie(\til{H})}$.  Let $\mathbf{p}_{\til{H}}$ be a nonzero vector in $\bar{\mathbf{p}}_{\til{H}}$. Then $\rho(a)\mathbf{p}_{\til{H}}=\chi(a)\mathbf{p}_{\til{H}}$ for every $a\in A$. Let $g\in X(H_0,U)$ with $\pi(g)\in Y_0$. Then $\til{H}=gH_0g^{-1}$, and hence $\rho(g^{-1})\mathbf{p}_{\til{H}}\in\bar{\mathbf{p}}_{H_0}$.
Since $Y$ is compact, there is a compact subset $M\subset G$ such that $A\pi(g)\subset Y\subset M\pi(e)$. This implies $A\subset M\Gamma g^{-1}$. Thus
$$\chi(A)\mathbf{p}_{\til{H}}=\rho(A)\mathbf{p}_{\til{H}}\subset\rho(M\Gamma g^{-1})\mathbf{p}_{\til{H}}.$$
By \cite[Thm. 3.4]{DM}, $\rho(\Gamma g^{-1})\mathbf{p}_{\til{H}}$ is discrete. So $\rho(M\Gamma g^{-1})\mathbf{p}_{\til{H}}$ is a closed subset of $V$ and does not contain $0$. Hence $0\notin\cl{\chi(A)\mathbf{p}_{\til{H}}}$. This implies that $\chi(A)\ne\R^*$. So $\chi$ is trivial. Hence $A\subset N_G^1(\til{H})$.

Let $W$ be the unipotent radical of the Zariski closure of $\til{H}$ in $G$. Since $A\subset N_G^1(\til{H})$, we have $W\subset\til{H}$ and $A\subset N_G^1(W)$.
Since the real algebraic group $AW$ is solvable and $\R$-split, by Borel's fixed point theorem \cite[Prop. 15.2]{Borel}, $AW$ is contained in a
minimal parabolic subgroup $P$ of $G$. The set
$\Phi^+=\{\alpha\in\Phi:U_\alpha\subset P\}$
is a system of positive roots. So there exists $a_0\in A$ such that
$\alpha(a_0)>1$ for every $\alpha\in\Phi^+$. If $W$ is nontrivial, then all eigenvalues of $\Ad(a_0)|_{\Lie(W)}$ are of the form $\alpha(a_0)$ $(\alpha\in\Phi^+)$, and hence $\det\Ad(a_0)|_{\Lie(W)}>1$. This contradicts $A\subset N_G^1(W)$. So $W$ is trivial. Hence $\til{H}$ is reductive. The proof of (i) is completed.

\textbf{Step 6.}  We now prove (ii). Since $\til{H}$ is reductive, so is $L$ (see e.g. \cite[Lem 1.1]{LR}), and we have $L=N_G^1(\til{H})^0$. Let $y\in Y_0$,  $g\in X(H_0,U)$ with $\pi(g)=y$. Then $$Ly=N_G^1(\til{H})^0y=gN_G^1(H_0)^0\pi(e).$$  By \cite[Thm. 3.4]{DM}, $N_G^1(H_0)\pi(e)$ is closed. So $Ly$ is also closed. Since $A\subset L$, we have $Y = \cl{Ay} \subset Ly$. It remains to prove that $Ly$ is of finite volume. If not, then by Proposition \ref{prop: reductive hom spaces}, there is a nontrivial $\R$-diagonalizable connected algebraic central subgroup $Z$ of $L$ such that $Zy$ is a divergent orbit in $G/\Gamma$. Since $L$ contains $A$ and is reductive, we must have $Z\subset A$. Thus $Zy\subset Y$. This contradicts the compactness of $Y$.
\end{proof}

\begin{proof}[Proof of Theorem \ref{T:main1}]
Firstly, we remark that if $S$ is a Lie subgroup of $G$ containing $A$, then $S$ has no nontrivial connected unipotent subgroups if and only
$S^0\subset Z_G(A)$. In fact, since $S$ contains $A$, we have
$$\Lie(S)=(\Lie(S)\cap\Lie(Z_G(A)))\oplus\bigoplus_{\alpha\in\Phi}(\Lie(S)\cap\mathfrak{u_\alpha}).$$
If $S$ has no nontrivial connected unipotent subgroups, then for every $\alpha\in\Phi$ we have $\Lie(S)\cap\mathfrak{u_\alpha}=0$. Thus $\Lie(S)\subset\Lie(Z_G(A))$, and hence $S^0\subset Z_G(A)$. Conversely, since $Z_G(A)$ consists of semisimple elements, if $S^0\subset Z_G(A)$ then $S$ has no nontrivial connected unipotent subgroups.

We now prove that $\Stab_G(Y)$ has no nontrivial connected unipotent subgroups by induction on $\dim G$. If $\dim G=0$,
there is nothing to prove. Assume $\dim G>0$ and the assertion holds
for groups of smaller dimension. We first prove the following:

\medskip

\textbf{Claim.} If there exists a nontrivial connected normal algebraic
subgroup $N$ of $G$ such that $\pi(\Gamma)$ is a lattice in $G/N$,
where $\pi:G\to G/N$ is the projection, then any connected unipotent subgroup $U$ of $\Stab_G(Y)$ is contained in $N$.

\medskip

Let $G'=G/N$. Then $G'$ is a reductive real algebraic group, and $\pi(A)$ is a
maximal connected $\R$-diagonalizable subgroup of $G'$. Let
$\bar{\pi}:G/\Gamma\to G'/\pi(\Gamma)$ be the induced projection. Then
$\bar{\pi}(Y)$ is a compact $\pi(A)$-minimal subset of
$G'/\pi(\Gamma)$, and $\pi(U)$ is a connected unipotent subgroup of $\Stab_{G'}(\bar{\pi}(Y))$. By the induction hypothesis, $\pi(U)$ is trivial. Hence $U\subset N$. This proves the claim.

\medskip

We conclude the proof by considering three cases.

\textbf{Case 1.} Suppose $G$ is not semisimple. By \cite[Cor. 8.27]{Rag}, the group
$N=C(G)^0$ satisfies the assumption of the claim. So any connected unipotent subgroup $U$ of $\Stab_G(Y)$ is contained in $C(G)$. But $C(G)$ consists of semisimple elements. So $U$ must be trivial.

\textbf{Case 2.} Suppose $G$ is semisimple and $\Gamma$ is
reducible. Then there are nontrivial connected normal subgroups
$G_1,G_2$ of $G$ such that $G=G_1G_2$, $G_1\cap G_2$ is discrete, and
the assumption of the claim is satisfied for $N=G_1,G_2$. Thus any connected unipotent subgroup of $\Stab_G(Y)$ is contained in $(G_1\cap G_2)^0$, which is trivial.

\textbf{Case 3.} Suppose $G$ is semisimple and $\Gamma$ is
irreducible. Suppose, to the contrary, that $\Stab_G(Y)$ has a nontrivial connected unipotent subgroup $U$. Since $\Stab_G(Y)$ contains $A$, we may assume that $U$ is normalized by $A$. By Proposition \ref{prop: minimal normalized}, there
exist $y\in Y$ and a connected reductive subgroup $\til{H}$ of $G$
containing $U$ such that $\cl{Uy}=\til{H}y$, and such
that $Ly$ is closed of finite volume and contains $Y$, where
$L=N_G(\til{H})^0$ is a reductive real algebraic group containing $A\til{H}$. There are
three subcases:
\begin{itemize}
  \item[(i)] $L\ne G$. The fact $U\subset\Stab_L(Y)$ contradicts
    the induction hypothesis.
  \item[(ii)] $L=G$ but $\til{H}\ne G$. Then $\til{H}$ is a proper normal subgroup
    of $G$, and has a closed orbit in $G/\Gamma$ of finite
    volume. This contradicts the irreducibility of $\Gamma$.
  \item[(iii)] $\til{H}=G$. Then $Y\supset\cl{Uy}=\til{H}y=G/\Gamma$. Hence $A$ acts minimally on $G/\Gamma$. But
    by \cite[Thm. 2.8]{PR}, if $C$ is Cartan subgroup of $G$ containing $A$, then there are compact $C$-orbits in
    $G/\Gamma$, a contradiction.
\end{itemize}
This completes the proof of Theorem \ref{T:main1}.
\end{proof}

We conclude this section with an example\footnote{The idea of constructing this example is due to Uri Shapira.}. Let $G$ be the underlying real algebraic group of $\SL_2(\C)$. Consider the order $\OO=\Z[\sqrt{-d}]$ in the imaginary quadratic field $\Q(\sqrt{-d})$, where $d>0$ is a square-free integer. Then $\Gamma=\SL_2(\OO)$ is a lattice in $G$. The subgroup $A\subset G$ consisting of positive real diagonal matrices is a maximal connected $\R$-diagonalizable subgroup, and the subgroup $C\subset G$ consisting of all diagonal matrices is a Cartan subgroup. Given a compact $C$-orbit $Cy$ in $G/\Gamma$, the $A$-action on $Cy$ admits two possibilities: either $A$ acts minimally on $Cy$, or all $A$-orbits in $Cy$ are closed. In what follows, we classify all compact $C$-orbits on which $A$ acts minimally. Note that the stabilizer of such a $C$-orbit always contains $C$. Thus the conclusion of Theorem \ref{T:main1} cannot be replaced by the stronger statement `$\Stab_G(Y)$ has no connected subgroups not contained in $A$'.

We first parameterize compact $C$-orbits in $G/\Gamma$. Let $\gamma\in\Gamma$ be a semisimple element of infinite order. Then there exists $g\in G$ with $g\gamma g^{-1}\in C$. It follows that $C\cap g\Gamma g^{-1}$ is infinite. This implies that $C\pi(g)$ is compact, where $\pi:G\to G/\Gamma$ is the projection. Note that $C\pi(g)$ is determined by $\gamma$ up to left translation by
$\begin{bmatrix}
  0&-1\\1&0
\end{bmatrix}$.
By abuse of language, we say that $C\pi(g)$ is \emph{the} compact orbit defined by $\gamma$. 
Every compact $C$-orbit arises in this way. In fact, if $C\pi(g)$ is compact, then $C\cap g\Gamma g^{-1}$ is infinite. Any element $\gamma\in\Gamma$ such that $g\gamma g^{-1}$ lies in $C$ and is of infinite order defines $C\pi(g)$.

The compact $C$-orbits which are $A$-minimal can be characterized as follows.

\begin{prop}\label{example}
Let $\gamma\in\Gamma$ be a semisimple element of infinite order. The following statements are equivalent:
\begin{itemize}
\item[(i)] The $A$-action on the compact $C$-orbit defined by $\gamma$ is not minimal.
\item[(ii)] There is a positive integer $n$ such that $\gamma^n$ has real eigenvalues.
\item[(iii)] The extension field of $\Q$ generated by an eigenvalue of $\gamma$ is Galois over $\Q$.
\item[(iv)] There exists $k\in\{0,1,2,3,4\}$ such that the Pell equation
$$(4-k)a^2-kdb^2=k(4-k)$$
holds, where $a,b\in\Z$ are such that $\mathrm{tr}(\gamma)=a+b\sqrt{-d}$.
\end{itemize}
\end{prop}

\begin{proof}
(i) $\Longleftrightarrow$ (ii): Let $C\pi(g)$ be the compact $C$-orbit defined by $\gamma$. Then $g\gamma g^{-1}\in C$. Note that the cyclic group $\langle g\gamma g^{-1}\rangle$ is infinite, hence is of finite index in $C\cap g\Gamma g^{-1}$. It follows that
\begin{align*}
&\text{$C\pi(g)$ is not $A$-minimal} \\
\Longleftrightarrow \ & A\pi(g) \text{ is compact }\\
\Longleftrightarrow \ & A\cap g\Gamma g^{-1} \text{ is infinite }\\
\Longleftrightarrow \ & A\cap\langle g\gamma g^{-1}\rangle \text{ is infinite }\\
\Longleftrightarrow \ & \text{there is a positive integer $n$ such that } g\gamma^n g^{-1}\in A \\
\Longleftrightarrow \ & \text{there is a positive integer $n$ such that $\gamma^n$ has real eigenvalues.}
\end{align*}
This proves the equivalence of (i) and (ii).

\medskip

In order to prove the equivalence of (ii), (iii), and (iv), we consider the polynomial
\begin{align*}
F(x)&=(x^2-\mathrm{tr}(\gamma)x+1)(x^2-\overline{\mathrm{tr}(\gamma)}x+1)\\
&=x^4-2ax^3+(a^2+b^2d+2)x^2-2ax+1,
\end{align*}
where $a,b$ are as in (iv). Let $\lambda$ be an eigenvalue of $\gamma$. Then the roots of $F$ are $\{\lambda,\lambda^{-1},\bar{\lambda},\bar{\lambda}^{-1}\}$.
It is easy to see that if $b=0$ then (ii), (iii), and (iv) hold. In what follows, we assume that $b\ne0$. Then $\lambda$ is neither real nor a root of unity.
We claim that in this case $F$ is irreducible over $\Q$. In fact, otherwise it would follow from $\lambda\notin\R$ that $\Q(\lambda)$ is an imaginary quadratic field, and hence $\lambda$, as an algebraic unit in $\Q(\lambda)$, would be a root of unity, a contradiction. The irreducibility of $F$ implies that $[\Q(\lambda):\Q]=4$. Let $K$ be the splitting field of $F$ over $\Q$. We identify $\Gal(K/\Q)$ with a permutation group on the roots of $F$. Then $\Gal(K/\Q)$ is contained in the group $D$ of permutations that send $\{\lambda,\lambda^{-1}\}$ onto $\{\lambda,\lambda^{-1}\}$ or $\{\bar{\lambda},\bar{\lambda}^{-1}\}$.
The group $D$ is a dihedral group of order $8$. So $[K:\Q]=4$ or $8$. Note that $\Q(\lambda)$ is Galois over $\Q$ if and only if $[K:\Q]=4$.
Note also that $\Gal(K/\Q)$ contains the complex conjugation, which we denote by $c$. 

\medskip

(ii) $\Longrightarrow$ (iii): Suppose to the contrary that $\Q(\lambda)$ is not Galois over $\Q$. Then $[K:\Q]=8$, and hence $\Gal(K/\Q)=D$. There is a nontrivial element $\sigma\in\Gal(K/\Q)$ of order $2$ that pointwise fixes $\Q(\lambda)$, which must be the transposition of $\bar{\lambda}$ and $\bar{\lambda}^{-1}$. The subfield $\Q(\lambda)\cap\R$ is pointwise fixed by $\sigma$ and $c$. It is easy to see that $\sigma$ and $c$ generate $D$. So $\Q(\lambda)\cap\R=\Q$. It follows from (ii) that $\lambda^n\in\Q(\lambda)\cap\R=\Q$ for some $n\in\N$. As a rational algebraic unit, $\lambda^n$ must be $\pm1$. So $\lambda$ is a root of unity, a contradiction.

\medskip

(iii) $\Longrightarrow$ (iv): It follows from (iii) that $|\Gal(K/\Q)|=4$. So $\Gal(K/\Q)$ consists of even permutations in $D$ (which form the only order $4$ subgroup of $D$ containing $c$), and hence the discriminant $\Delta(F)$ of $F$ is a square.
It is easy to show that 
$$\Delta(F)=16b^4d^2((a+2)^2+b^2d)((a-2)^2+b^2d)$$ 
(see \cite[Example 13.1.3]{Co}). The case of $k\in\{0,4\}$ in (iv) corresponds to $ab=0$. Suppose $ab\ne0$. Then $$(a+2)^2+b^2d=km^2, \qquad (a-2)^2+b^2d=kn^2$$ for some positive integers $k,m,n$ with $k$ square-free and $m\ne n$. It follows that
\begin{align}
8a&=k(m+n)(m-n),\label{1}\\
4(a^2+b^2d+4)&=k(m+n)^2+k(m-n)^2.\label{2}
\end{align}
By eliminating $a$, we obtain
$$(k(m+n)^2-16)(k(m-n)^2-16)+64b^2d=0.$$
Thus
\begin{equation}\label{3}
k(m-n)^2<16.
\end{equation}
Since $k$ is square-free, it follows from \eqref{1} that $m-n$ is even. In view of \eqref{3}, we have $m-n=\pm2$, and hence $k\in\{1,2,3\}$. Now \eqref{1} becomes $k(m+n)=\pm 4a$. It follows from \eqref{2} that
$$k(a^2+b^2d+4)=4a^2+k^2.$$
This is the Pell equation in (iv).

\medskip

(iv) $\Longrightarrow$ (ii): The case of $k\in\{0,4\}$ is obvious. Suppose $k\in\{1,2,3\}$. It is straightforward to check from (iv) that
$$\lambda=\left(\frac{a}{\sqrt{k}}+\frac{\sqrt{d}b}{\sqrt{4-k}}\right)\frac{\sqrt{k}+\sqrt{k-4}}{2}$$
is an eigenvalue of $\gamma$.
Note that $\frac{\sqrt{k}+\sqrt{k-4}}{2}=e^{\frac{\pi i}{3}},e^{\frac{\pi i}{4}},$ or $e^{\frac{\pi i}{6}}$ when $k=1,2,$ or $3$. Thus (ii) follows.
\end{proof}

\section{Proof of Theorem \ref{T:main2}}

The root system $\Phi$ may be non-reduced. For $\alpha\in\Phi$, denote $[\alpha]=\{c\alpha:c>0\}\cap\Phi$. Then $[\alpha]$ has three possibilities: $\{\alpha\}$, $\{\alpha,2\alpha\}$, and $\{\alpha,\alpha/2\}$. Let $\goth{u}_{[\alpha]}=\sum_{\beta\in[\alpha]}\goth{u}_\beta$, and let $U_{[\alpha]}=\exp(\goth{u}_{[\alpha]})$ be the unipotent group with Lie algebra $\goth{u}_{[\alpha]}$. We first prove:

\begin{lem}\label{L:bi-invariant}
Let $R\subset G$ be a closed subset invariant under the conjugation of $A$ such that $R\cap U_{[\alpha]}\subset\{e\}$ for every $\alpha\in\Phi$. Then
there exists a neighborhood $\Omega$ of $e$ in $G$ such that $R\cap\Omega\subset Z_G(A)$.
\end{lem}

\begin{proof}
Suppose the conclusion of the lemma is not true. Then there exists a sequence $\mathbf{r}_n\in\exp^{-1}(R)\sm\Lie(Z_G(A))$ with $\mathbf{r}_n\to0$. Since $\goth{g}=\Lie(Z_G(A))\oplus\bigoplus_{\alpha\in\Phi}\goth{u}_\alpha$, we can write
$$\mathbf{r}_n=\mathbf{z}_n+\sum_{\alpha\in\Phi}\mathbf{u}_{n,\alpha},$$
where $\mathbf{z}_n\in\Lie(Z_G(A))$, $\mathbf{u}_{n,\alpha}\in\goth{u}_\alpha$. Note that $\mathbf{z}_n$ and $\mathbf{u}_{n,\alpha}$ converge to $0$ as $n\to\infty$.
Identifying a character $A\to\R^*$ with its differential, we can think of a root $\alpha$ as an element in $\goth{a}^*$.
By passing to a subsequence, we may assume that there exists
$\alpha_0\in\Phi$ such that
$$|\mathbf{u}_{n,\alpha}|^{\frac{1}{|\alpha|}}\le
|\mathbf{u}_{n,\alpha_0}|^{\frac{1}{|\alpha_0|}}, \qquad \forall
\alpha\in\Phi, n\in\N,$$
where $|\mathbf{u}_{n,\alpha}|$ (resp. $|\alpha|$) is the norm of $\mathbf{u}_{n,\alpha}$ (resp. $\alpha$) with respect to a fixed inner product on $\goth{g}$ (resp. $\goth{a}^*$). Since $\mathbf{r}_n\notin\Lie(Z_G(A))$, we have $|\mathbf{u}_{n,\alpha_0}|>0$. Let $t_n\in\R$ be such that $e^{t_n}=|\mathbf{u}_{n,\alpha_0}|^{-\frac{1}{|\alpha_0|^{2}}}$, and let $\mathbf{a}\in\goth{a}$ be such that
$\alpha(\mathbf{a})=\langle\alpha,\alpha_0\rangle$ for every $\alpha\in\Phi$. Then
\begin{align*}
\Ad(\exp(t_n\mathbf{a}))\mathbf{r}_n&=\mathbf{z}_n+\sum_{\alpha\in\Phi}e^{t_n\alpha(\mathbf{a})}\mathbf{u}_{n,\alpha}\\
&=\mathbf{z}_n+|\mathbf{u}_{n,\alpha_0}|^{-1}\mathbf{u}_{n,\alpha_0}
+\sum_{\alpha\in\Phi\sm\{\alpha_0\}}|\mathbf{u}_{n,\alpha_0}|^{-\frac{\langle\alpha,\alpha_0\rangle}{|\alpha_0|^2}}\mathbf{u}_{n,\alpha}.
\end{align*}
By passing to a subsequence, we may assume that $$|\mathbf{u}_{n,\alpha_0}|^{-1}\mathbf{u}_{n,\alpha_0}\to\mathbf{u}_1$$ for some $\mathbf{u}_1\in\goth{u}_{\alpha_0}$ with $|\mathbf{u}_1|=1$. For $\alpha\in\Phi\sm\{\alpha_0\}$, we have
$$|\mathbf{u}_{n,\alpha_0}|^{-\frac{\langle\alpha,\alpha_0\rangle}{|\alpha_0|^2}}|\mathbf{u}_{n,\alpha}|
\le|\mathbf{u}_{n,\alpha_0}|^{\frac{|\alpha||\alpha_0|-\langle\alpha,\alpha_0\rangle}{|\alpha_0|^2}}.$$
If $\alpha\in\Phi\sm[\alpha_0]$, then $\langle\alpha,\alpha_0\rangle<|\alpha||\alpha_0|$, and hence
$$|\mathbf{u}_{n,\alpha_0}|^{-\frac{\langle\alpha,\alpha_0\rangle}{|\alpha_0|^2}}\mathbf{u}_{n,\alpha}\to0.$$
If $[\alpha_0]\ne\{\alpha_0\}$ and $\alpha\in[\alpha_0]\sm\{\alpha_0\}$, then $\langle\alpha,\alpha_0\rangle=|\alpha||\alpha_0|$, and we have
$$|\mathbf{u}_{n,\alpha_0}|^{-\frac{\langle\alpha,\alpha_0\rangle}{|\alpha_0|^2}}|\mathbf{u}_{n,\alpha}|\le1.$$
Hence by passing to a further subsequence, we may assume that
$$|\mathbf{u}_{n,\alpha_0}|^{-\frac{\langle\alpha,\alpha_0\rangle}{|\alpha_0|^2}}\mathbf{u}_{n,\alpha}\to\mathbf{u}_2$$ for some $\mathbf{u}_2\in\goth{u}_{\alpha}$. In summary, a subsequence of $\Ad(\exp(t_n\mathbf{a}))\mathbf{r}_n$ converges to a nonzero element $\mathbf{u}\in\goth{u}_{[\alpha_0]}$, where $\mathbf{u}=\mathbf{u}_1$ if $[\alpha_0]=\{\alpha_0\}$, and $\mathbf{u}=\mathbf{u}_1+\mathbf{u}_2$ if $[\alpha_0]\ne\{\alpha_0\}$. Since $\exp^{-1}(R)$ is closed
and $\Ad(A)$-invariant, we have $\mathbf{u}\in\exp^{-1}(R)$. Thus
$e\ne\exp(\mathbf{u})\in R\cap U_{[\alpha_0]}$, a contradiction.
\end{proof}

\begin{lem}\label{L:cocompact}
Let $Y\subset G/\Gamma$ be a compact $A$-minimal set, $\alpha\in\Phi$,
and $Y_\alpha\subset Y$ be an $A_\alpha$-minimal set. Suppose
$AY_\alpha=Y$. Then $\Stab_A(Y_\alpha)$ is cocompact in $A$.
\end{lem}

\begin{proof}
We first prove that for any open subset $B$ of $A$, $BY_\alpha$ is
open in $Y$. It suffices to prove that for every $b\in B$, $bY_\alpha$
is contained in the interior $\mathrm{int}(BY_\alpha)$ of
$BY_\alpha$. Let $C$ be a compact neighborhood of $e$ in $A$ such that
$bC^{-1}C\subset B$, and let $\{a_n\}$ be a sequence in $A$ such that
$\bigcup_{n=1}^\infty a_nC=A$. Then $\bigcup_{n=1}^\infty
a_nCY_\alpha=Y$. By the Baire category theorem, some $a_nCY_\alpha$,
and hence $CY_\alpha$, has an interior point. Let $c\in C$ be such
that $\mathrm{int}(CY_\alpha)\cap cY_\alpha\ne\emptyset$. Then
$$\mathrm{int}(BY_\alpha)\cap bY_\alpha\supset\mathrm{int}(bc^{-1}CY_\alpha)\cap bY_\alpha=bc^{-1}(\mathrm{int}(CY_\alpha)\cap cY_\alpha)\ne\emptyset.$$
Since $\mathrm{int}(BY_\alpha)$ is open $A_\alpha$-invariant and $bY_\alpha$ is $A_\alpha$-minimal, we have $bY_\alpha\subset\mathrm{int}(BY_\alpha)$. Thus $BY_\alpha$ is open in $Y$.

Now we prove that $\Stab_A(Y_\alpha)$ is cocompact in $A$. Since $\Stab_A(Y_\alpha)\supset A_\alpha$, it suffices to prove that $\Stab_A(Y_\alpha)\ne A_\alpha$.
Let $B_1\subset B_2\subset\cdots$ be a nested sequence of bounded open subsets of $A$ such that $\bigcup_{n=1}^\infty B_n=A$. Then $\{B_nY_\alpha\}$ is an open cover of $Y$. Since $Y$ is compact, there exists $n_0\in\N$ such that $Y=B_{n_0}Y_\alpha$. Let $a\in A$ be such that $a^{-1}B_{n_0}\cap A_\alpha=\emptyset$. Since $aY_\alpha$ is $A_\alpha$-minimal and $Y=B_{n_0}Y_\alpha$, there exists $b\in B_{n_0}$ with $aY_\alpha=bY_\alpha$. It follows that $a^{-1}b\in\Stab_A(Y_\alpha)$. But $a^{-1}b\notin A_\alpha$. So $\Stab_A(Y_\alpha)\ne A_\alpha$. This proves the lemma.
\end{proof}

We are now prepared to prove Theorem \ref{T:main2}.

\begin{proof}[Proof of Theorem \ref{T:main2}]
The assumption implies that every $A_\alpha$-minimal subset of $Y$ is of the form $aY_\alpha$ ($a\in A$), hence by Lemma \ref{L:cocompact}, has a cocompact stabilizer in $A$. Let $$R=\{g\in G:gY\cap Y\ne\emptyset\}.$$
Then $R$ is closed and invariant under the conjugation of $A$. We first prove that $R\cap U_{[\alpha]}=\{e\}$ for every $\alpha\in\Phi$. Suppose the contradiction. Then there exists $e\ne u\in U_{[\alpha]}$ such that $uY\cap Y\ne\emptyset$. Since $u$ commutes with $A_\alpha$, the compact set $uY\cap Y$ is $A_\alpha$-invariant, hence contains an $A_\alpha$-minimal set $Y'$. Since $\Stab_A(Y')$ is cocompact in $A$, there exists a sequence $\{a_n\}\subset\Stab_A(Y')$ such that $a_nua_n^{-1}\to e$. Note that $u^{-1}Y'\subset Y$ is also $A_\alpha$-minimal. So $u^{-1}Y'=aY'$ for some $a\in A$, i.e., $ua\in\Stab_G(Y')$. It follows that
$$\Stab_G(Y')\ni(ua)a_n(ua)^{-1}a_n^{-1}=u(a_nua_n^{-1})^{-1}\to u.$$
By the closedness of $\Stab_G(Y')$, we have $u\in\Stab_G(Y')$. So $a_nua_n^{-1}\in\Stab_G(Y')$. This implies that the group $U_{[\alpha]}\cap\Stab_G(Y')$ is non-discrete. So $U\df(U_{[\alpha]}\cap\Stab_G(Y'))^0$ is nontrivial. On the other hand, since $U_{[\alpha]}$ and $\Stab_G(Y')$ are normalized by $\Stab_A(Y')$, so is $U$. But $N_G(U)$ is Zariski closed in $G$ and $\Stab_A(Y')$ is Zariski dense in $A$. So $A$ normalizes $U$. Hence
$$UY=UAY'=AUY'=AY'=Y.$$
It follows that $U\subset\Stab_G(Y)$. This conflicts Theorem \ref{T:main1}. Hence $R\cap U_{[\alpha]}=\{e\}$ for every $\alpha\in\Phi$.

Let $Z=Z_G(A)^0$. Then $Z$ is the direct product of $A$ and a connected compact subgroup $M$ of $G$ centralized by $A$. Let $y\in Y$. We claim that $Zy$ is compact and contains $Y$. By Lemma \ref{L:bi-invariant} and the preceding paragraph, there exists an open neighborhood $\Omega$ of $e$ in $G$ such that $R\cap\Omega\subset Z$. It follows that $Y\cap\Omega y\subset(R\cap\Omega)y\subset Zy$. Since $Y$ is $A$-minimal and $Y\cap\Omega y$ is a neighborhood of $y$ in $Y$, we have $Y=A(Y\cap\Omega y)\subset Zy$. This in turn implies that $Zy=ZY=MY$ is compact, proving the claim.

Let $\Lambda=\{g\in Z:gy=y\}$. Then $\Lambda$ is a lattice in $Z$ and the natural map $Z/\Lambda\to Zy$ is a homeomorphism. Let $p:Z=A\times M\to M$ be the projection, and let $T=(\overline{p(\Lambda)})^0$. By \cite[Thm. 8.24]{Rag}, $T$ is solvable, hence is a compact torus. Note that
$AT=(\overline{Ap(\Lambda)})^0=(\overline{A\Lambda})^0$. So $ATy$ is closed and contains $Ay$ as a dense subset. It follows that $Y=ATy$. This completes the proof.
\end{proof}

\end{document}